\documentclass[11pt]{amsart}

\usepackage{amsmath}
\usepackage{amsfonts}
\usepackage{amssymb}
\usepackage{fullpage}

\usepackage{hyperref}
\hypersetup{colorlinks,linkcolor={red},citecolor={blue},urlcolor={blue}}
\newtheorem*{theorem*}{Theorem}
\newtheorem{theorem}{Theorem}[section]
\newtheorem{lemma}[theorem]{Lemma}
\newtheorem{corollary}[theorem]{Corollary}
\newtheorem{proposition}[theorem]{Proposition}

\theoremstyle{definition}
\newtheorem{definition}[theorem]{Definition}

\theoremstyle{remark}
\newtheorem{remark}[theorem]{Remark}

\numberwithin{equation}{section}

\newcommand{\CC}{\mathbb C}
\newcommand{\HH}{\mathbb H}

\newcommand{\NN}{\mathbb N}

\newcommand{\QQ}{\mathbb Q}
\newcommand{\RR}{\mathbb R}
\newcommand{\ZZ}{\mathbb Z}
\newcommand{\SL}{\mathop{\mathrm {SL}}\nolimits}

\newcommand{\Sp}{\mathop{\mathrm {Sp}}\nolimits}

\newcommand{\latt}[1]{{\langle{#1}\rangle}}
\newcommand{\Ord}{\mathop{\mathrm {Ord}}\nolimits}
\newcommand{\m}{\operatorname{mod}}
\newcommand{\w}{\mathrm{w}}
\newcommand{\Tr}{\operatorname{Tr}}

\newenvironment{psmallmatrix}
  {\left(\begin{smallmatrix}}
{\end{smallmatrix}\right)}

\newcommand{\abs}[1]{\lvert#1\rvert}
\begin{document}

\title[Orbits of Jacobi forms and  Theta relations]{Orbits of Jacobi forms and  Theta relations}

\author{Valery Gritsenko}

\address{Laboratoire Paul Painlev\'{e}, Universit\'{e} de Lille and NRU HSE, Moscow}

\email{valery.gritsenko@univ-lille.fr}

\author{Haowu Wang}

\address{School of Mathematics and Statistics, Wuhan University, Wuhan 430072, Hubei, China}

\email{haowu.wangmath@whu.edu.cn}


\date{\today}


\begin{abstract}
Jacobi theta functions with rational characteristics can be viewed as vector-valued Jacobi forms. Theta relations usually correspond to different constructions of certain Jacobi forms. 
From this observation, we extract a new approach, which is called orbits of Jacobi forms, to produce identities on Jacobi theta functions. Our approach not only provides simple proofs of many known theta relations but also produces a large number of new identities, which can be considered as generalizations of Riemann's theta relations.  
\end{abstract}

\maketitle

\section{Introduction}
The Jacobi theta function with rational characteristic defined as
\begin{equation*}
\vartheta_{a,b}(\tau,z)
=\sum_{n\in \ZZ}e^{\pi i(n+a)^2\tau+2\pi i(n+a)(z+b)}
\end{equation*}
is a holomorphic function on $\HH \times\CC$ and plays a vital role in many areas \cite{Mu}, including the constructions of modular forms, the research of Abelian varieties and moduli spaces, and the theory of quadratic forms. There are a lot of interesting and useful identities on Jacobi theta functions, which are called theta relations. In this paper we introduce a new method to construct theta relations. Our approach is based on the theory of Jacobi forms. 

Jacobi forms \cite{EZ} are holomorphic functions on $\HH\times \CC$ which are modular in the first variable $\tau\in \HH$ and quasi doubly-periodic in the second variable $z\in \CC$. The automorphic group of Jacobi forms is the Jacobi group which is the semidirect product of $\SL_2(\ZZ)$ with the integral Heisenberg group. In particular, the real Heisenberg group acts on a Jacobi form of weight $k$ and index $t$ via 
\begin{equation*}
\left(\varphi\lvert_{k,t}[x,y;r]\right)(\tau,z)= e^{2\pi i t ( x^2\tau +2xz+xy+r)}  \varphi (\tau, z+ x \tau + y).
\end{equation*}
It is known \cite{GN} that the Jacobi triple product formula
\begin{equation*}
\vartheta(\tau,z) = q^{\frac{1}{8}}(\zeta^{\frac{1}{2}}- \zeta^{-\frac{1}{2}}) \prod_{n\geq 1} (1-q^n\zeta)(1-q^n \zeta^{-1})(1-q^n)
\end{equation*}
defines a holomorphic Jacobi form of  weight $\frac{1}{2}$ and index $\frac{1}{2}$ with a 
multiplier system of order $8$, where $q=e^{2\pi i\tau}$ and $\zeta=e^{2\pi iz}$. All Jacobi 
theta functions  with rational characteristics can be expressed by $\vartheta$. More precisely, 
$\vartheta_{a+\frac{1}{2},b+\frac{1}{2}}$ equals $\vartheta\lvert_{\frac{1}{2},\frac{1}{2}}
[a,b;0]$ up to a constant factor. From which we observe that Jacobi theta functions with certain characteristics constitute a vector-valued function. We analyze in \S \ref{sec:2} the 
behavior of this vector with respect to the full Jacobi modular group. Note that the 
quadratic form in the Jacobi theta functions is the odd form $x^2$.
Theta relations should come from different constructions of some Jacobi forms in terms of the vectors of theta-functions with characteristics. We next give an accurate description of this observation. We first embed the set 
\begin{equation*}
R_N:=\left( \frac{1}{N}, \frac{1}{N} \right) \cdot \SL_2(\ZZ)/ \ZZ^2=\left\lbrace \left(\frac{u}{N},\frac{v}{N}\right)\in\frac{1}{N}\ZZ^2 : 0\leq u,v <N, \, (u,v)\neq (0,0) \right\rbrace
\end{equation*}
in the  rational  Heisenberg group.
Then the vector-valued function $(\vartheta\lvert_{\frac{1}{2},\frac{1}{2}} Y)_{Y\in R_N}$, which is called {\it the orbit} of $\vartheta$, transforms like a vector-valued Jacobi form with additional non-regular correcting factors (see (\ref{orbit1}), (\ref{orbit2})).  Therefore, we are able to construct Jacobi forms by taking special combinations of vector-valued functions  $(\vartheta \lvert_{\frac{1}{2},\frac{1}{2}} Y)(\tau,0)$ and $(\vartheta\lvert_{\frac{1}{2},\frac{1}{2}} Y)(\tau,mz)$ with $m\geq 1$.
A general example of such combinations is given below.

\begin{theorem*}[see Theorem \ref{MTH}]
Let $N\geq 2$ and $a$, $b$, $c$, $d$ be non-negative integers satisfying that $N \lvert (b+2c)$ and $a+b+c+d \equiv 0 \m 2N$. Then the function
\begin{equation*}
\Tr_{a,b,c,d}^{(N)}:=\sum_{Y\in R_N} (\vartheta\lvert Y)^a(\tau,0)(\vartheta\lvert Y)^b(\tau,z)(\vartheta\lvert Y)^c(\tau,2z)(\vartheta\lvert Y)^d(\tau,Nz)
\end{equation*}
is a holomorphic Jacobi form of weight $\frac{a+b+c+d}{2}$ and index $\frac{b+4c+N^2d}{2}$ with a character of finite order.
\end{theorem*}

Given parameters, if $\Tr_{a,b,c,d}^{(N)}=0$ then we get a theta relation. If $\Tr_{a,b,c,d}^{(N)}\neq 0$, then it is possible to find basic Jacobi forms, such as Jacobi--Eisenstein series or nice generators of a certain space of Jacobi forms, to represent $\Tr_{a,b,c,d}^{(N)}$, because the corresponding space of Jacobi forms is finite dimensional. In this case, we obtain a generalized theta relation. By applying the above theorem to different parameters, more than fifty identities are derived (see Corollaries \ref{coro1}, \ref{coro2}, \ref{coro3}).  

In section \ref{sec:4} we also establish formulas to represent holomorphic Jacobi forms of weight $6$ and weak Jacobi forms of weight $0$ in terms of Jacobi theta functions of order $2$ (see Proposition \ref{prop02} and Corollary \ref{cor:weak}).  

The paper is organized as follows. In section \ref{sec:2}, we recall some basics on Jacobi forms of integral and half-integral index, as well as Jacobi theta functions. We also explain the action of the real Heisenberg group on Jacobi forms. Section \ref{sec:3} is devoted to the introduction of orbits of Jacobi forms. In section \ref{sec:4}, we apply the technique of orbits of $\vartheta$ to the constructions of Jacobi forms and of theta relations. In section \ref{sec:5}, we present several applications of our formulas.

\section{Jacobi forms and Jacobi theta functions}\label{sec:2}

\subsection{Definition of Jacobi forms} 
In this subsection we introduce Jacobi forms of integral and half-integral index. We refer to \cite{EZ, GJ, GN} for more details. 

Jacobi forms appear naturally as Fourier--Jacobi coefficients of Siegel modular forms of genus 2. In what follows we use this fact to explain the action of the real Jacobi group on the space of holomorphic functions on $\HH\times \CC$. Let 
$$ \HH_2=\left\{Z\in M_2(\CC) : Z={Z^t},\; \mathrm{Im}(Z)>0 \right\}$$
be the Siegel upper-half space of genus 2, and let $\Sp_2(\RR)$ be the real symplectic group of rank 2. We define the standard slash operator $\lvert_k$ ($k\in \frac{1}{2}\ZZ$) on the space of functions on $\HH_2$ as 
$$ (F\lvert_k M)(Z):= \det (CZ+D)^{-k}F(M\latt{Z})$$
where 
$$
M=\left(\begin{array}{cc}
A & B \\ 
C & D
\end{array}  \right) \in \Sp_2(\RR) \quad \text{and} \quad M\latt{Z}:=(AZ+B)(CZ+D)^{-1}.
$$

Let us consider the real Jacobi group $\Gamma^J(\RR)=\Gamma_{\infty}(\RR)/ \{\pm I_4\}$, where $\Gamma_{\infty}(\RR)$ is a maximal parabolic subgroup of the real symplectic group defined by
$$ \Gamma_{\infty}(\RR):=\left\{ \left(\begin{array}{cccc}
* & 0 & * & * \\ 
* & * & * & * \\ 
* & 0 & * & * \\ 
0 & 0 & 0 & *
\end{array}  \right) \in \Sp_2(\RR) \right\}.$$ 
The homomorphism
\begin{equation}
A=\left( \begin{array}{cc}
a & b \\ 
c & d
\end{array}  \right) \in \SL_2(\RR) \longmapsto  \{A\}=\left(\begin{array}{cccc}
a & 0 & b & 0 \\ 
0 & 1 & 0 & 0 \\ 
c & 0 & d & 0 \\ 
0 & 0 & 0 & 1
\end{array}  \right)\in \Gamma^J(\RR)
\end{equation}
defines an embedding of $\SL_2(\RR)$ into $\Gamma^J(\RR)$. 
The second standard subgroup of $\Gamma^J(\RR)$ is the real Heisenberg group  
$$ H(\RR):=\{ [x,y;r]: x,y, r \in \RR\}$$
where
\begin{equation}\label{Heis}
[x,y;r]=\left(\begin{array}{cccc}
1 & 0 & 0 & y \\ 
x & 1 & y & r \\ 
0 & 0 & 1 & -x \\ 
0 & 0 & 0 & 1
\end{array}  \right).
\end{equation}
The multiplication in $H(\RR)$ is given by 
\begin{equation}
[x_1,y_1;r_1] \cdot [x_2,y_2;r_2]= [x_1+x_2,y_1+y_2;r_1+r_2+x_1y_2-x_2y_1] .
\end{equation}
The group $ \SL_2(\RR)$ acts on $H(\RR)$ via
\begin{equation}
A\cdot [x,y;r]=\{A\}[x,y;r]\{A^{-1}\}=[(x,y)A^{-1};r]= [dx-cy,ay-bx;r].
\end{equation}
The real Jacobi group $\Gamma^J(\RR)$ is the semidirect product $\SL_2(\RR)\ltimes H(\RR)$ with the multiplication 
\begin{equation}
(A,h_1)\cdot (B,h_2)=(AB, (B^{-1}\cdot h_1)\cdot h_2).
\end{equation}
Let $\Gamma^J(\ZZ)=\SL_2(\ZZ)\ltimes H(\ZZ)$ be the integral Jacobi group, where 
$$ H(\ZZ):=\{ [x,y;r]: x,y, r \in \ZZ\}$$
is the integral Heisenberg group. 
We describe the characters of the integral Jacobi group. Let $\chi : \Gamma^J(\ZZ) \rightarrow \CC^\times$ be a character (or a multiplier system) of finite order. By \cite{CG13}, its restriction to $\SL_2(\ZZ)$ is $\upsilon_{\eta}^D$, where $D\in \NN$ and $\upsilon_{\eta}$ is the multiplier system (or a projective character) of order $24$ of the Dedekind $\eta$-function
$$
\eta(\tau)=q^{\frac{1}{24}}\prod_{n=1}^\infty(1-q^n).
$$
Moreover, one has
\begin{equation}
\chi ( \{A\} \cdot [x,y;r])= \upsilon_{\eta}^D (A) \cdot \upsilon_{H}^\varepsilon([x,y;r]),
\end{equation}
where $\varepsilon=0,1$ and $\upsilon_{H}$ is the unique binary character of $H(\ZZ)$ defined as
\begin{equation}
\upsilon_{H}([x,y;r])=(-1)^{x+y+xy+r}.
\end{equation}

We now introduce the action of the Jacobi group on the space of functions on $\HH\times \CC$. Fix
$$
Z=\left(\begin{array}{cc}
\tau & z \\ 
z & \omega
\end{array}  \right) \in \HH_2, \quad q=e^{2\pi i\tau}, \quad \zeta=e^{2\pi i z}.
$$
Let $k\in \frac{1}{2}\ZZ$ and $t\in \frac{1}{2}\NN$. Assume that 
$\varphi(\tau,z)$ is a holomorphic function on $\HH\times \CC$. Then 
\begin{equation}
\widetilde{\varphi}(Z):= \varphi(\tau,z)\exp(2\pi itw)
\end{equation}
is a holomorphic function on $\HH_2$. The action of the Jacobi group $\Gamma^J(\RR)$ on $\varphi(\tau,z)$ is defined as
\begin{equation}
(\varphi\lvert_{k,t}g)(\tau,z):=(\widetilde{\varphi}\lvert_k g)(Z)\exp(-2\pi i t \omega).
\end{equation}
In particular, the generators of $\Gamma^J(\RR)$ act on $\varphi(\tau,z)$ via
\begin{align}
\big(\varphi\lvert_{k,t}\{A\}\big)(\tau,z)&=(c\tau + d)^{-k} e^{-2\pi i t \tfrac{cz^2}{c \tau + d}}\varphi \left( \frac{a\tau +b}{c\tau + d},\frac{z}{c\tau + d} \right),\\
\big(\varphi\lvert_{k,t}[x,y;r]\big)(\tau,z)&= e^{2\pi i t ( x^2\tau +2xz+xy+r)}  \varphi (\tau, z+ x \tau + y).
\end{align}

\begin{definition}\label{def:JF}
Let $k\in \frac{1}{2}\ZZ$ and $t\in \frac{1}{2}\NN$. A holomorphic function $\varphi(\tau,z)$ on $\HH\times \CC$ is called a {\it weak Jacobi form of weight $k$ and index $t$} with a multiplier system (or a character) $\chi: \Gamma^J(\ZZ) \rightarrow \CC^\times$ if it satisfies the functional equation 
$$ (\varphi\lvert_{k,t}g)(\tau,z)= \chi(g)\varphi(\tau,z), \quad \forall g \in \Gamma^J(\ZZ),$$
and if it has a Fourier expansion of the form
$$ \varphi(\tau,z)=\sum_{\substack{ n\geq 0,\, n\equiv \frac{D}{24} \m \ZZ\\ l\in \frac{1}{2}\ZZ}} f(n,l)\exp\big(2\pi i(n\tau+lz)\big),$$
where $0\leq D<24$ is given by $\chi\lvert_{\SL_2(\ZZ)} = \upsilon_{\eta}^D$.
If $\varphi$ further satisfies the condition
$$ f(n,l) \neq 0 \Longrightarrow 4nt - l^2 \geq 0 $$
then it is called a {\it holomorphic Jacobi form}. If $\varphi$ further satisfies the stronger condition
$$ f(n,l) \neq 0 \Longrightarrow 4nt - l^2 > 0 $$
then it is called a {\it Jacobi cusp form}.
\end{definition}

We denote by $J^{\w}_{k,t}(\chi)$ the finite dimensional vector space of all weak Jacobi forms of weight $k$, index $t$ and character $\chi$. The corresponding spaces of holomorphic Jacobi forms and Jacobi cusp forms are denoted by $J_{k,t}(\chi)$ and $J^{\mathrm{cusp}}_{k,t}(\chi)$, respectively. If the character is trivial, we also write $J_{k,t}=J_{k,t}(1)$ for short.

We remark that if $\varphi\in J_{k,t}^{\w}(\chi)$ is nonzero then $\chi=\upsilon_\eta^D\cdot \upsilon_H^{2t}$. By definition, we have 
$$
J^{\w}_{k,t}(\upsilon_{\eta}^D\cdot\upsilon_{H}^{2t})=\eta^D\cdot J^{\w}_{k-\frac{D}{2},t}(\upsilon_{H}^{2t}),
\quad\text{and}\quad k-D/2 \in \ZZ. 
$$

\subsection{Jacobi theta functions}
Let $(a,b)\in (\QQ^2\backslash \ZZ^2)\cup \{(0,0)\}$.  The Jacobi theta function with rational characteristic $(a,b)$ (see \cite[Chapter 1]{Mu}) is defined as   
\begin{equation}\label{thetaab}
\vartheta_{a,b}(\tau,z)
=\sum_{n\in \ZZ}e^{\pi i(n+a)^2\tau+2\pi i(n+a)(z+b)}.
\end{equation}
It is a holomorphic function on $\HH\times \CC$. We set
$$
\theta_{a,b}(\tau)=\vartheta_{a,b}(\tau,0).
$$
In the notation, the classical Jacobi theta functions of order $2$ have the following form  
\begin{align*}
&\vartheta_{00}=\vartheta_{0,0}& &\vartheta_{01}=\vartheta_{0, \frac{1}2}& &\vartheta_{10}=\vartheta_{\frac 1{2}, 0}& &\vartheta_{11}=\vartheta_{\frac 1{2}, \frac{1}2}&\\
&\theta_{00}=\theta_{0,0}& &\theta_{01}=\theta_{0, \frac{1}2}& &\theta_{10}=\theta_{\frac 1{2}, 0}&
\end{align*}
For $ab=00$, $01$, $10$ we also put 
$$
\xi_{ab}(\tau,z)=\vartheta_{ab}(\tau,z)/\theta_{ab}(\tau).
$$

Let $q=e^{2\pi i\tau}$ and $\zeta=e^{2\pi iz}$. The Jacobi triple product formula
\begin{equation}
\vartheta(\tau,z) = q^{\frac{1}{8}} (\zeta^{\frac{1}{2}}-\zeta^{-\frac{1}{2}}) \prod_{n\geq 1} (1-q^n\zeta)(1-q^n \zeta^{-1})(1-q^n) 
\end{equation}
defines a holomorphic Jacobi form of  weight $\frac{1}{2}$ and index $\frac{1}{2}$
with a multiplier system of order $8$, that is, $\vartheta\in J_{\frac{1}{2},\frac{1}{2}}(\upsilon_{\eta}^3 \cdot \upsilon_{H})$ (see \cite{GN}). Note that $\vartheta(\tau,z) =-i\vartheta_{11}(\tau,z)$ and it vanishes precisely on 
$$
\{ (\tau,z)\in \HH\times \CC : z\in \ZZ\tau+\ZZ \}.
$$
Moreover, these zeros are all simple.

The Jacobi theta functions with rational characteristics can be expressed by means of $\vartheta$.
\begin{align}
\vartheta(\tau,z+a\tau+b)=&\exp\Big(-\pi i \big(a^2\tau+2az+2ab+a+\frac{1}{2}\big)\Big)\vartheta_{a+\frac{1}{2},b+\frac{1}{2}}(\tau,z),\\
\label{eq:orbit-theta}\Big(\vartheta\lvert_{\frac{1}{2},\frac{1}{2}} [a,b;0]\Big)(\tau,z)=&\exp\Big(-\pi i \big(ab+a+\frac{1}{2}\big)\Big)
\vartheta_{a+\frac{1}{2},b+\frac{1}{2}}(\tau,z),\\
\vartheta_{a+x,b+y}(\tau,z)=&\exp\bigl(2\pi i ay\bigr)\vartheta_{a,b}(\tau,z), 
\quad \forall \, x,y \in\ZZ.
\end{align}

\subsection{The space of Jacobi forms}
For convenience, we introduce some notation of modular forms and Jacobi forms. The function $E_{2k}=1+O(q)$ denotes the $\SL_2(\ZZ)$-Eisenstein series of weight $2k$ for $k\geq 1$. The function $\Delta=\eta^{24}$ is the cusp form of weight 12 on $\SL_2(\ZZ)$. The function $E_{2k,m}$ denotes the Jacobi--Eisenstein series of weight $2k$ and index $m$ for $k,m \geq 1$ introduced in \cite[\S 2]{EZ}. The functions $\phi_{0,1}$, $\phi_{0,2}$, $\phi_{0,3}$, $\phi_{0,4}$ are generators of the graded ring of weak Jacobi forms of weight $0$ with integral Fourier coefficients and their indices are respectively $1$, $2$, $3$, $4$ (see \cite{G1, G2, GW2}). The function $\phi_{-2,1}$ is the unique weak Jacobi form of weight $-2$ and index $1$ up to constant (see \cite{EZ}). The following constructions can be found in \cite{EZ, G2, GW2}.
\begin{align}
\phi_{0,1}(\tau,z)&=\frac{(E_4^2E_{4,1}-E_6E_{6,1})(\tau,z)}{144\Delta(\tau)}=\zeta+\zeta^{-1}+10+O(q)\in J_{0,1}^{\w},\\
\phi_{0,2}(\tau,z)&= 2\left[(\xi_{00}\xi_{01})^2+(\xi_{00}\xi_{10})^2+(\xi_{10}\xi_{01})^2\right](\tau,z)=\zeta+\zeta^{-1}+4+O(q)\in J_{0,2}^{\w},\\
\phi_{0,3}(\tau,z)&=\frac{\vartheta^2(\tau,2z)}{\vartheta^2(\tau,z)}=\zeta+\zeta^{-1}+2+O(q)\in J_{0,3}^{\w},\\
\phi_{0,4}(\tau,z)&= \frac{\vartheta(\tau,3z)}{\vartheta(\tau,z)}=\zeta+\zeta^{-1}+1+O(q)\in J_{0,4}^{\w},\\
\phi_{-2,1}(\tau,z)&= \frac{\vartheta^2(\tau,z)}{\eta^6(\tau)}=\zeta+\zeta^{-1}-2+O(q)\in J_{-2,1}^{\w}.
\end{align}

For ease of use later, we recall the structure of the space of Jacobi forms. From \cite[Theorem 9.4]{EZ}, we know that the bigraded ring of all weak Jacobi forms of even weight and integral index is a polynomial algebra in $\phi_{0,1}$ and $\phi_{-2,1}$ over the ring of $\SL_2(\ZZ)$ modular forms. Moreover, for $k\in \ZZ$ and $t\in \NN$, we have
\begin{align}
&J_{2k+1,t}^\w=\phi_{-1,2}\cdot J_{2k+2,t-2}^\w,& &\phi_{-1,2}(\tau,z)=\frac{\vartheta(\tau,2z)}{\eta^3(\tau)}\in J_{-1,2}^\w.&
\end{align}
Let $k\in\ZZ$ and $t\in \NN$.  By \cite[Lemma 1.4]{G1}, we have
\begin{align}
&J_{2k,t+\frac{1}{2}}^\w=\phi_{0,\frac{3}{2}}\cdot J_{2k,t-1}^\w,& &J_{2k+1,t+\frac{1}{2}}^\w=\phi_{-1,\frac{1}{2}}\cdot J_{2k+2,t}^\w,&
\end{align}
where $\phi_{0,\frac{3}{2}}$ and $\phi_{-1,\frac{1}{2}}$ are defined as
\begin{align}
&\phi_{0,\frac{3}{2}}(\tau,z)=\frac{\vartheta(\tau,2z)}{\vartheta(\tau,z)}\in J_{0,\frac{3}{2}}^\w,& &\phi_{-1,\frac{1}{2}}(\tau,z)=\frac{\vartheta(\tau,z)}{\eta^3(\tau)}\in J_{-1,\frac{1}{2}}^\w.&
\end{align}

\section{Orbits of Jacobi forms}\label{sec:3}
In this section we introduce the notion of orbits of Jacobi forms.  Let us fix $X=(p,q)\in \QQ^2$. We consider the set
\begin{equation}
X \cdot \SL_2(\ZZ)=\{X M: M\in \SL_2(\ZZ) \},
\end{equation}
where $X M$ is the usual matrix multiplication. 
We choose a finite system of representatives 
\begin{equation}
R_X=X\cdot  \SL_2(\ZZ)/\ZZ^2.
\end{equation}
For arbitrary $Y_1\in R_X$ and $M\in \SL_2(\ZZ)$, we can choose $Y_2\in R_X$ such that $Y_1 M-Y_2\in \ZZ^2$. We compute in the real Heisenberg group \eqref{Heis}
\begin{align*}
[Y_1;0]\cdot \{M\}&=\{M\}\cdot(\{M^{-1}\}[Y_1;0]\{M\})\\
&=\{M\}\cdot [Y_1M;0]\\ 
&=\{M\}\cdot[Y_1M-Y_2;0][0,0;\det(Y_2;Y_1M)][Y_2;0],
\end{align*}
where $\det(Y_2;Y_1M)$ is the determinant of the matrix whose first row is $Y_2$ and second row is $Y_1M$.

Let $\varphi \in J_{k,t}(\upsilon_{\eta}^D \cdot \upsilon^{2t}_{H})$. From discussions above, we derive the transformation formula under the action of $\SL_2(\ZZ)$ as follows
\begin{equation}\label{orbit1}
\big(\varphi\lvert_{k,t}[Y_1;0]\big)\lvert_{k,t}\{M\}=\upsilon_{\eta}^D(M)\upsilon^{2t}_{H}\big([Y_1M-Y_2;0]\big)\exp\big(2\pi it \det(Y_2;Y_1M)\big)\varphi\lvert_{k,t}[Y_2;0].
\end{equation}

Let $Y=(y_1,y_2)\in R_X$. For any $h=[h_1,h_2;0]\in H(\ZZ)$, we have 
$$
[Y;0]\cdot h=h\cdot[0,0;2(y_1h_2-h_1y_2)]\cdot[Y;0],
$$
which yields the transformation formula under the action of the integral Heisenberg group
\begin{equation}\label{orbit2}
\big(\varphi\lvert_{k,t}[Y;0]\big)\lvert_{k,t}h=\upsilon^{2t}_{H}(h)\exp\big(4\pi i t(y_1h_2-h_1y_2)\big) \varphi\lvert_{k,t}[Y;0].
\end{equation}

Let $f(n,l)$ denote the Fourier coefficients of $\varphi$. The order of $\varphi$ is defined as
\begin{equation}
\Ord(\varphi)=\min\{4nt-l^2: f(n,l)\neq 0\}.
\end{equation}
Then 
\begin{equation}\label{ord}
 \varphi\lvert_{k,t}[Y;0]\equiv 0 \quad \text{or}\quad   \Ord(\varphi\lvert_{k,t}[Y;0])\geq \Ord(\varphi).
\end{equation}
The following two formulas will be used later. Their proofs are straightforward.
\begin{align}
\label{orbit3}\big[\big(\varphi\lvert_{k,t}[Y;0]\big)(\tau,mz)\big]\lvert_{k,m^2t}\{M\}&=\big(\varphi\lvert_{k,t}[Y;0]\lvert_{k,t}\{M\}\big)(\tau,mz),\\
\label{orbit4}\big[\big(\varphi\lvert_{k,t}[Y;0]\big)(\tau,mz)\big]\lvert_{k,m^2t}h&=\big(\varphi\lvert_{k,t}[Y;0]\lvert_{k,t}h'\big)(\tau,mz),
\end{align} 
where $m$ is a positive integer and $h=[h_1,h_2;r]\in H(\ZZ)$, 
$h'=[mh_1,mh_2;m^2r]$.
 
From now on, we abbreviate $\varphi\lvert_{k,t}[Y;0]$ to $\varphi\lvert Y$ 
when there is no confusion.  Equations \eqref{orbit1} and 
\eqref{orbit2} describe the transformations of the vector-valued function 
$\left( \varphi\lvert Y \right)_{Y\in R_X}$ with respect to the integral Jacobi 
group. We call this vector-valued function the \textit{orbit} of 
$\varphi$. We note that the transformation formulae depend on the choice 
of the system representatives $R_X$. This paper aims to construct new 
Jacobi forms by orbits of basic Jacobi forms. 
We consider the special combinations of two types of functions: 
$\left(\varphi \lvert Y \right) (\tau,0)$ and 
$\left(\varphi\lvert Y \right)(\tau,mz)$, $m\geq 1$. Here, the adjective 
\textit{special} means that by this combination one can cancel the bad 
additional factors coming from 
$\upsilon^{2t}_{H}([Y_1M-Y_2;0])\exp\bigl(2\pi it \det(Y_2;Y_1M)\bigr)$ 
and 
$\exp\bigl(4\pi i t(y_1h_2-h_1y_2)\bigr)$ in (\ref{orbit1}) 
and (\ref{orbit2}).

\section{Theta relations and their generalizations}\label{sec:4}
Let us consider $X=(\frac{1}{N},\frac{1}{N})$, where $N\geq 2$ is a positive integer. Then we fix
$$
R_N=\left( \frac{1}{N}, \frac{1}{N} \right) \cdot \SL_2(\ZZ)/ \ZZ^2=\left\lbrace \left(\frac{u}{N},\frac{v}{N}\right)\in\frac{1}{N}\ZZ^2 : 0\leq u,v <N, \, (u,v)\neq (0,0) \right\rbrace.
$$
In this section, we use the orbit $(\vartheta\lvert Y)_{Y\in R_N}$ of $\vartheta$ to construct many Jacobi forms. When the constructed Jacobi form is identically zero, we get an identity on Jacobi theta functions with rational characteristics, which is in fact a theta relation. When the constructed Jacobi form is not zero, we obtain a generalized Riemann's theta relation. For instance, we get formulae that represent basic Jacobi forms (Jacobi--Eisenstein series, generators of the ring of Jacobi forms, ...) in terms of Jacobi theta functions with rational characteristics.

We now introduce our first result. The product $\prod_{Y\in R_N}\vartheta\lvert Y$ of the functions in the orbit of $\vartheta$  transforms like a Jacobi form. More precisely, we prove the following proposition.

\begin{proposition}\label{prop01}
Let $N\geq 2$ be a positive integer. Then we have the following identity 
\begin{equation}\label{J01}
\prod_{\substack{0\leq u,v <N\\ (u,v)\neq (0,0)}}\vartheta_{\frac{u}{N}+\frac{1}{2},\frac{v}{N}+\frac{1}{2}}(\tau,z)=(-1)^{N-1}\eta^{N^2-1}(\tau)\frac{\vartheta(\tau,Nz)}{\vartheta(\tau,z)}.
\end{equation}
In particular, let $z=0$, we get
\begin{equation}\label{J02}
\prod_{\substack{0\leq u,v <N\\ (u,v)\neq (0,0)}}\theta_{\frac{u}{N}+\frac{1}{2},\frac{v}{N}+\frac{1}{2}}(\tau)=(-1)^{N-1} N \eta^{N^2-1}(\tau).
\end{equation}
\end{proposition}
\begin{proof}
Recall from \eqref{eq:orbit-theta} that $\vartheta_{\frac{u}{N}+\frac{1}{2},\frac{v}{N}+\frac{1}{2}}$ equals $\vartheta|Y$ up to constant factor for $Y=(u/N,v/N)$. 
Although we can compute the transformations of $\prod_{Y\in R_N}\vartheta\lvert Y$ under the action of the integral Jacobi group, the computations are very complicated. To pass through this difficulty, we consider the function $\prod_{Y\in R_N}\vartheta^{2N} \lvert Y$ because the additional bad factors in (\ref{orbit1}) and (\ref{orbit2}) will be cancelled.
For any $M\in \SL_2(\ZZ)$ and any $h\in \ZZ^2$, we have 
\begin{align*}
(\vartheta^{2N}\lvert Y) \lvert M&=\upsilon_{\eta}^{6N}(M)\vartheta^{2N}\lvert Y',\\
(\vartheta^{2N}\lvert Y) \lvert [h;0]&=\vartheta^{2N}\lvert Y,
\end{align*}
where $Y'\in R_N$ such that $Y'-YM \in \ZZ^2$. Therefore,
$$
\Phi(\tau,z) =  \prod_{\substack{0\leq u,v <N\\ (u,v)\neq (0,0)}}\vartheta_{\frac{u}{N}+\frac{1}{2},\frac{v}{N}+\frac{1}{2}}^{2N}(\tau,z)=c\prod_{Y\in R_N}\vartheta^{2N}\lvert Y \in J_{N^3-N,N^3-N}(\upsilon_{\eta}^{2(N^3-N)}),
$$
where $c$ is a constant.
We know from \cite[Lemma 3]{GD} that $\prod_{\substack{0\leq u,v <N\\ (u,v)\neq (0,0)}}\vartheta_{\frac{u}{N}+\frac{1}{2},\frac{v}{N}+\frac{1}{2}}(\tau,0)$ has the lead term of $q$-expansion: $(-1)^{N-1}Nq^{\frac{N^2-1}{24}}$. It follows that $\frac{\Phi}{\eta^{2(N^3-N)}}\in J_{0,N^3-N}^\w$. Moreover, $\frac{\Phi}{\eta^{2(N^3-N)}}$ has the same zeros as $\left(\dfrac{\vartheta(\tau,Nz)}{\vartheta(\tau,z)} \right)^{2N}$ and their zeros have the same multiplicity. We hence obtain 
$$ \Big(\Phi / \eta^{2(N^3-N)}\Big) / \left(\frac{\vartheta(\tau,Nz)}{\vartheta(\tau,z)} \right)^{2N} \in J_{0,0}^\w.$$ 
Therefore, the above function is constant, which clearly forces that the function
$$
\left(\Big(\prod\vartheta_{\frac{u}{N}+\frac{1}{2},\frac{v}{N}+\frac{1}{2}}(\tau,z)\Big) / \eta^{N^2-1}(\tau) \right) / \left(\frac{\vartheta(\tau,Nz)}{\vartheta(\tau,z)}\right)
$$
is also constant. We then prove the desired formula by comparing the lead terms of $q$-expansion.
\end{proof}

\begin{corollary} By taking $N=2,3$, we deduce the following identities:
\begin{align}
&\label{ap01} \vartheta_{01}\vartheta_{10}\vartheta_{00}=\eta^3 \frac{\vartheta(\tau,2z)}{\vartheta(\tau,z)},\\
& \label{c1}\theta_{01}\theta_{10}\theta_{00}=2\eta^3,\\
& \vartheta_{\frac{1}{6},\frac{1}{6}}\vartheta_{\frac{1}{6},\frac{1}{2}}\vartheta_{\frac{1}{6},\frac{5}{6}}\vartheta_{\frac{1}{2},\frac{1}{6}}\vartheta_{\frac{1}{2},\frac{5}{6}}\vartheta_{\frac{5}{6},\frac{1}{6}}\vartheta_{\frac{5}{6},\frac{1}{2}}\vartheta_{\frac{5}{6},\frac{5}{6}}=-\eta^8 \frac{\vartheta(\tau,3z)}{\vartheta(\tau,z)},\\
& \theta_{\frac{1}{6},\frac{1}{6}}\theta_{\frac{1}{6},\frac{1}{2}}\theta_{\frac{1}{6},\frac{5}{6}}\theta_{\frac{1}{2},\frac{1}{6}}\theta_{\frac{1}{2},\frac{5}{6}}\theta_{\frac{5}{6},\frac{1}{6}}\theta_{\frac{5}{6},\frac{1}{2}}\theta_{\frac{5}{6},\frac{5}{6}}=-3 \eta^8.
\end{align}
\end{corollary}

\begin{remark}
The formula (\ref{c1}) is classical (see \cite{Mu}). The formula (\ref{J02}) is crucial to the construction of elliptic units and was first proved in \cite{S} using Kronecker's second limit formula.  An explicit expression of (\ref{J02}) can also be found in \cite{GD}. The formula (\ref{J01}) is a generalization of (\ref{J02}) and seems to be new.
\end{remark}

We now present our main theorem. 

\begin{theorem}\label{MTH}
Let $N\geq 2$ and $a$, $b$, $c$, $d$ be non-negative integers satisfying $a+b+c+d \equiv 0 \m 2N$ and $N \lvert (b+2c)$. Then the function
\begin{equation}\label{Tr}
\Tr_{a,b,c,d}^{(N)}=\sum_{Y\in R_N} (\vartheta\lvert Y)^a(\tau,0)(\vartheta\lvert Y)^b(\tau,z)(\vartheta\lvert Y)^c(\tau,2z)(\vartheta\lvert Y)^d(\tau,Nz)
\end{equation}
is a holomorphic Jacobi form. More precisely, 
$$
\Tr_{a,b,c,d}^{(N)}\in J_{\frac{a+b+c+d}{2}, \frac{b+4c+N^2d}{2}}(\upsilon_{\eta}^{3(a+b+c+d)}\cdot \upsilon_{H}^{b+Nd}).
$$
\end{theorem}

\begin{proof}
For convenience we set $F_Y=(\vartheta\lvert Y)^a(\tau,0)(\vartheta\lvert Y)^b(\tau,z)(\vartheta\lvert Y)^c(\tau,2z)(\vartheta\lvert Y)^d(\tau,Nz)$. From (\ref{orbit1}) and (\ref{orbit3}) we conclude that
\begin{align*}
F_{Y_1}\lvert M&=\upsilon_{\eta}^{3l}(M) \upsilon_{H}^{l}([Y_1M-Y_2;0])\cdot e^{\pi il \det(Y_2;Y_1M) }F_{Y_2}\\
&=\upsilon_{\eta}^{3l}(M)F_{Y_2}
\end{align*}
for any $M\in \SL_2(\ZZ)$, where $l=a+b+c+d$. The above equality holds because $2N \lvert l$ and $N\det(Y_2;Y_1M)\in \ZZ$. By (\ref{orbit2}) and (\ref{orbit4}), we find that
\begin{align*}
F_{Y}\lvert h&=\upsilon_{H}^{b+Nd}(h)e^{2\pi i (b+2c+Nd)(y_1h_2-h_1y_2)}F_{Y}\\
&=\upsilon_{H}^{b+Nd}(h)F_{Y}
\end{align*}
for any $h\in H(\ZZ)$, because $N(y_1h_2-h_1y_2)\in \ZZ$. We thus prove that $\Tr_{a,b,c,d}^{(N)}$ is a holomorphic Jacobi form by the explanations above and (\ref{ord}). 
\end{proof}

By selecting different parameters in (\ref{Tr}), we get a lot of identities. Firstly, our main theorem reproduces many known identities. 

\begin{corollary}\label{coro1}
\begin{align}
&\label{c00}\Tr^{(2)}_{4,0,0,0}:& &\theta_{01}^4+\theta_{10}^4-\theta_{00}^4=0,\\
&\label{c01} \Tr^{(2)}_{2,2,0,0}:&  &\theta_{01}^2\vartheta_{01}^2+\theta_{10}^2\vartheta_{10}^2-\theta_{00}^2\vartheta_{00}^2=0, \\
&\Tr^{(2)}_{0,4,0,0}:& &\vartheta_{01}^4+\vartheta_{10}^4-\vartheta_{00}^4=\vartheta^4,\\
&\Tr^{(2)}_{3,0,1,0}:& &\theta_{01}^3\vartheta_{01}(\tau,2z)+\theta_{10}^3\vartheta_{10}(\tau,2z)-\theta_{00}^3\vartheta_{00}(\tau,2z)=-2\vartheta(\tau,z)^4,\\
&\Tr^{(2)}_{1,2,1,0}:& &\theta_{01}\vartheta_{01}(\tau,2z)\vartheta_{01}(\tau,z)^2+\theta_{10}\vartheta_{10}(\tau,2z)\vartheta_{10}(\tau,z)^2-\theta_{00}\vartheta_{00}(\tau,2z)\vartheta_{00}(\tau,z)^2=0,\\
&\Tr^{(2)}_{8,0,0,0}:& &E_4=\frac{1}{2}\left(\theta_{01}^8+\theta_{10}^8+\theta_{00}^8\right)=  \theta_{00}^4\theta_{01}^4+\theta_{00}^4\theta_{10}^4-\theta_{10}^4\theta_{01}^4.
\end{align}
\end{corollary}
Secondly, Theorem \ref{MTH} produces many new identities on Jacobi 
theta functions of order 2.

\begin{corollary}\label{coro2}
\begin{align}
&\label{c02} \Tr^{(2)}_{6,2,0,0}:&  &\theta_{01}^6\vartheta_{01}^2+\theta_{10}^6\vartheta_{10}^2+\theta_{00}^6\vartheta_{00}^2=2E_{4,1},\\
&\Tr^{(2)}_{10,2,0,0}:& &\theta_{01}^{10}\vartheta_{01}^2+\theta_{10}^{10}\vartheta_{10}^2- \theta_{00}^{10}\vartheta_{00}^2=-4\eta^{12}\phi_{0,1},\\
&\Tr^{(2)}_{14,2,0,0}:& &\theta_{01}^{14}\vartheta_{01}^2+\theta_{10}^{14}\vartheta_{10}^2+\theta_{00}^{14}\vartheta_{00}^2=2E_{8,1},\\
&\label{ap05}\Tr^{(2)}_{4,4,0,0}:& &\theta_{01}^4\vartheta_{01}^4+\theta_{10}^4\vartheta_{10}^4+\theta_{00}^4\vartheta_{00}^4=2E_{4,2},\\
&\Tr^{(2)}_{2,6,0,0}:& &\theta_{01}^2\vartheta_{01}^6+\theta_{10}^2\vartheta_{10}^6+\theta_{00}^2\vartheta_{00}^6=2E_{4,3},\\
&\label{ap06}\Tr^{(2)}_{0,8,0,0}:& &\vartheta_{01}^8+\vartheta_{10}^8+\vartheta_{00}^8=2E_{4,4}+\vartheta^8,\\
&\Tr^{(2)}_{7,0,1,0}:& &\theta_{01}^7\vartheta_{01}(\tau,2z)+\theta_{10}^7\vartheta_{10}(\tau,2z)+\theta_{00}^7\vartheta_{00}(\tau,2z)=2E_{4,2},\\
&\Tr^{(2)}_{6,0,2,0}:& &\theta_{01}^6\vartheta_{01}(\tau,2z)^2+\theta_{10}^6\vartheta_{10}(\tau,2z)^2+\theta_{00}^6\vartheta_{00}(\tau,2z)^2=2E_{4,4}+2\vartheta^8,\\
&\Tr^{(2)}_{12,4,0,0}:& &\theta_{01}^{12}\vartheta_{01}^4+\theta_{10}^{12}\vartheta_{10}^4+\theta_{00}^{12}\vartheta_{00}^4=2E_{8,2}+\frac{272}{43} \eta^{12}\vartheta^4,\\
&\Tr^{(2)}_{15,0,1,0}:& &\theta_{01}^{15}\vartheta_{01}(\tau,2z)+\theta_{10}^{15}\vartheta_{10}(\tau,2z)+\theta_{00}^{15}\vartheta_{00}(\tau,2z)=2E_{8,2}+\frac{2336}{43}\eta^{12}\vartheta^4,\\
&\Tr^{(2)}_{10,6,0,0}:& &\theta_{01}^{10}\vartheta_{01}^6+\theta_{10}^{10}\vartheta_{10}^6+\theta_{00}^{10}\vartheta_{00}^6=2E_{8,3}-\frac{28}{547}\eta^{12}\vartheta^4\phi_{01},\\
&\Tr^{(2)}_{8,8,0,0}:& &\theta_{01}^{8}\vartheta_{01}^8+\theta_{10}^{8}\vartheta_{10}^8+\theta_{00}^{8}\vartheta_{00}^8=2E_{8,4}-\frac{73}{43}\eta^{12}\vartheta^4\phi_{0,2},\\
&\Tr^{(2)}_{13,2,1,0}:& &\theta_{01}^{13}\vartheta_{01}^2\vartheta_{01}(\tau,2z)+\theta_{10}^{13}\vartheta_{10}^2\vartheta_{10}(\tau,2z)+\theta_{00}^{13}\vartheta_{00}^2\vartheta_{00}(\tau,2z)\\
\nonumber & & &\quad =2E_{8,3}+\frac{2160}{547}\eta^{12}\vartheta^4\phi_{01},\\
&\Tr^{(2)}_{5,2,1,0}:& &\theta_{01}^5\vartheta_{01}(\tau,2z)\vartheta_{01}(\tau,z)^2+\theta_{10}^5\vartheta_{10}(\tau,2z)\vartheta_{10}(\tau,z)^2+\theta_{00}^5\vartheta_{00}(\tau,2z)\vartheta_{00}(\tau,z)^2\\
\nonumber & & &\quad =2E_{4,3},\\
&\Tr^{(2)}_{11,4,1,0}:& &\theta_{01}^{11}\vartheta_{01}^4\vartheta_{01}(\tau,2z)+\theta_{10}^{11}\vartheta_{10}^4\vartheta_{10}(\tau,2z)+\theta_{00}^{12}\vartheta_{00}^4\vartheta_{00}(\tau,2z)\\
\nonumber & & &\quad =2E_{8,4}+\frac{271}{43}\eta^{12}\vartheta^4\phi_{0,2}.
\end{align}
\end{corollary}

Theorem \ref{MTH} also yields a lot of interesting identities on Jacobi theta functions of order 3.

\begin{corollary}\label{coro3}
\begin{align}
&\label{c03}\Tr^{(3)}_{0,6,0,0}:& &\sum_{Y\in R_3} (\vartheta \lvert Y)^6 (\tau,z)=2\vartheta^6(\tau,z),\\
&\label{c04}\Tr^{(3)}_{6,0,0,0}& &\sum_{Y\in R_3} (\vartheta \lvert Y)^6 (\tau,0)=0,\\
&\Tr^{(3)}_{3,3,0,0}:& &\sum_{Y\in R_3} (\vartheta \lvert Y)^3 (\tau,0)(\vartheta \lvert Y)^3 (\tau,z)  =0,\\
&\Tr^{(3)}_{12,0,0,0}:& &\sum_{Y\in R_3} (\vartheta \lvert Y)^{12} (\tau,0)=-72\eta^{12}(\tau),\\
&\Tr^{(3)}_{9,3,0,0}:& &\sum_{Y\in R_3} (\vartheta \lvert Y)^{9} (\tau,0)(\vartheta \lvert Y)^{3} (\tau,z)=-36\eta^{12}(\tau) \frac{\vartheta(\tau,2z)}{\vartheta(\tau,z)},\\
&\label{c05}\Tr^{(3)}_{6,6,0,0}:& &\sum_{Y\in R_3} (\vartheta \lvert Y)^{6} (\tau,0)(\vartheta \lvert Y)^{6} (\tau,z)=-18\eta^{12}(\tau) \left( \frac{\vartheta(\tau,2z)}{\vartheta(\tau,z)} \right)^2,\\
&\Tr^{(3)}_{3,9,0,0}:& &\sum_{Y\in R_3} (\vartheta \lvert Y)^{3} (\tau,0)(\vartheta \lvert Y)^{9} (\tau,z)=-9\eta^{12}(\tau) \left( \frac{\vartheta(\tau,2z)}{\vartheta(\tau,z)} \right)^3,\\
&\Tr^{(3)}_{0,12,0,0}:& &\sum_{Y\in R_3} (\vartheta \lvert Y)^{12} (\tau,z)=-36\eta^{12}(\tau)  \frac{\vartheta(\tau,4z)}{\vartheta(\tau,2z)}+2\vartheta(\tau,z)^{12},\\
&\Tr^{(3)}_{21,3,0,0}:& &\sum_{Y\in R_3} (\vartheta \lvert Y)^{21} (\tau,0)(\vartheta \lvert Y)^{3} (\tau,z)=756\Delta \frac{\vartheta(\tau,2z)}{\vartheta(\tau,z)},\\
&\Tr^{(3)}_{24,0,0,0}& &\sum_{Y\in R_3} (\vartheta \lvert Y)^{24} (\tau,0)=1512\Delta,\\
&\Tr^{(3)}_{4,1,1,0}:& &\sum_{Y\in R_3} (\vartheta \lvert Y)^4 (\tau,0)(\vartheta \lvert Y)(\tau,z)(\vartheta \lvert Y) (\tau,2z)=0,\\
&\Tr^{(3)}_{1,4,1,0}:& &\sum_{Y\in R_3} (\vartheta \lvert Y) (\tau,0)(\vartheta \lvert Y)^4(\tau,z)(\vartheta \lvert Y) (\tau,2z)=0,\\
&\Tr^{(3)}_{2,2,2,0}:& &\sum_{Y\in R_3} (\vartheta \lvert Y)^2 (\tau,0)(\vartheta \lvert Y)^2(\tau,z)(\vartheta \lvert Y)^2 (\tau,2z)=0,\\
&\Tr^{(3)}_{0,3,3,0}:& &\sum_{Y\in R_3} (\vartheta \lvert Y)^3(\tau,z)(\vartheta \lvert Y)^3 (\tau,2z)=2\vartheta^3(\tau,z)\vartheta^3(\tau,2z),\\
&\Tr^{(3)}_{5,0,0,1}:& &\sum_{Y\in R_3} (\vartheta \lvert Y)^5(\tau,0)(\vartheta \lvert Y) (\tau,3z)=-3\vartheta^5(\tau,z)\vartheta(\tau,2z),\\
&\Tr^{(3)}_{3,1,1,1}:& &\sum_{Y\in R_3} (\vartheta \lvert Y)^3(\tau,0)(\vartheta \lvert Y) (\tau,z)(\vartheta \lvert Y) (\tau,2z)(\vartheta \lvert Y) (\tau,3z)=0,\\
&\Tr^{(3)}_{2,3,0,1}:& &\sum_{Y\in R_3} (\vartheta \lvert Y)^2(\tau,0)(\vartheta \lvert Y)^3 (\tau,z)(\vartheta \lvert Y) (\tau,3z)=-3\vartheta^4(\tau,z)\vartheta^2(\tau,2z),\\
&\Tr^{(3)}_{0,4,1,1}:& &\sum_{Y\in R_3} (\vartheta \lvert Y)^4(\tau,z)(\vartheta \lvert Y) (\tau,2z)(\vartheta \lvert Y) (\tau,3z)=-\vartheta^4(\tau,z)\vartheta(\tau,2z)\vartheta(\tau,3z),\\
&\Tr^{(3)}_{10,1,1,0}:& &\sum_{Y\in R_3} (\vartheta \lvert Y)^{10}(\tau,0)(\vartheta \lvert Y) (\tau,z)(\vartheta \lvert Y) (\tau,2z)=-3\eta^{12}\frac{\vartheta(\tau,2z)}{\vartheta(\tau,z)}\phi_{0,1},
\end{align}
\begin{align}
&\Tr^{(3)}_{11,0,0,1}:& &\sum_{Y\in R_3} (\vartheta \lvert Y)^{11}(\tau,0)(\vartheta \lvert Y) (\tau,3z)=-3\eta^{12}\frac{\vartheta(\tau,2z)}{\vartheta(\tau,z)}\left( \phi_{0,1}\phi_{0,2}- 15 \phi_{0,3}  \right),\\
&\Tr^{(3)}_{1,1,4,0}:& &\sum_{Y\in R_3} (\vartheta \lvert Y)(\tau,0)(\vartheta \lvert Y) (\tau,z)(\vartheta \lvert Y)^4 (\tau,2z)=3\vartheta^4(\tau,z)\vartheta(\tau,2z)\vartheta(\tau,3z),\\
&\Tr^{(3)}_{1,2,2,1}:& &\sum_{Y\in R_3} (\vartheta \lvert Y)(\tau,0)(\vartheta \lvert Y)^2 (\tau,z)(\vartheta \lvert Y)^2 (\tau,2z)(\vartheta \lvert Y)(\tau,3z)=0,\\
&\Tr^{(3)}_{7,4,1,0}:& &\sum_{Y\in R_3} (\vartheta \lvert Y)^7(\tau,0)(\vartheta \lvert Y)^4 (\tau,z)(\vartheta \lvert Y) (\tau,2z)=-24\eta^{12}\frac{\vartheta(\tau,3z)}{\vartheta(\tau,z)},\\
&\Tr^{(3)}_{15,3,0,0}:& &\sum_{Y\in R_3} (\vartheta \lvert Y)^{15}(\tau,0)(\vartheta \lvert Y)^3 (\tau,z)=3\eta^{6}E_6\frac{\vartheta(\tau,2z)}{\vartheta(\tau,z)},\\
&\Tr^{(3)}_{18,0,0,0}:& &\sum_{Y\in R_3} (\vartheta \lvert Y)^{18}(\tau,0)=6\eta^{6}E_6,\\
&\Tr^{(3)}_{27,3,0,0}:& &\sum_{Y\in R_3} (\vartheta \lvert Y)^{27}(\tau,0)(\vartheta \lvert Y)^3 (\tau,z)=-90\eta^{18}E_6\frac{\vartheta(\tau,2z)}{\vartheta(\tau,z)},\\
&\Tr^{(3)}_{30,0,0,0}:& &\sum_{Y\in R_3} (\vartheta \lvert Y)^{30}(\tau,0)=-180\eta^{18}E_6,\\
&\label{c06}\Tr^{(5)}_{10,0,0,0}:& &\sum_{Y\in R_5}(\vartheta \lvert Y)^{10} (\tau,0)=0,\\
&\Tr^{(7)}_{14,0,0,0}:& &\sum_{Y\in R_7}(\vartheta \lvert Y)^{14} (\tau,0)=0.
\end{align}
\end{corollary}

\begin{proof}[\textbf{Proof of Corollaries}]
Since the proofs of these identities are similar, we only give some of them.
\begin{enumerate}
\item Proof of (\ref{c01}):\\
We have $\Tr^{(2)}_{2,2,0,0}=\theta_{01}^2\vartheta_{01}^2+\theta_{10}^2\vartheta_{10}^2-\theta_{00}^2\vartheta_{00}^2\in J_{2,1}(\upsilon_{\eta}^{12})$. From $J_{2,1}(\upsilon_{\eta}^{12}) \subset \eta^{12}J^\w_{-4,1}$ and $J^\w_{-4,1}=\{0\}$, we conclude $J_{2,1}(\upsilon_{\eta}^{12})=\{0 \}$, which proves the desired identity.
\item Proof of $(\ref{c02})$:\\
We have $\Tr^{(2)}_{6,2,0,0}=\theta_{01}^6\vartheta_{01}^2+\theta_{10}^6\vartheta_{10}^2+\theta_{00}^6\vartheta_{00}^2\in J_{4,1}$. Since $J_{4,1}=\CC\cdot E_{4,1}$, we finish the proof by comparing their first Fourier coefficients.
\item Proof of (\ref{c03}) and (\ref{c04}):\\
Firstly, $\Tr^{(3)}_{0,6,0,0}=\sum_{Y\in R_3} (\vartheta \lvert Y)^6 (\tau,z) \in J_{3,3}(\upsilon_{\eta}^{18}) $. It is easy to see  
$$
J_{3,3}(\upsilon_{\eta}^{18})\subset \eta^{18} \cdot J_{-6,3}^\w=\CC \cdot \eta^{18} \phi_{-2,1}^3=\CC \cdot \vartheta^6(\tau,z).
$$
We hence prove the first formula by comparing the first Fourier coefficients. If we put $z=0$ then the second formula is obtained.
\item Proof of (\ref{c05}):\\
Similarly, we have 
$\Tr^{(3)}_{6,6,0,0}=\sum_{Y\in R_3} (\vartheta \lvert Y)^{6} (\tau,0)(\vartheta \lvert Y)^{6} (\tau,z) \in J_{6,3}(\upsilon_{\eta}^{12})$. Since
$$
J_{6,3}(\upsilon_{\eta}^{12})\subset \eta^{12}J_{0,3}^\w= \eta^{12}\latt{\phi_{0,1}^3, \phi_{0,1}\phi_{0,2},\phi_{0,3} }_{\CC}
$$
and $\eta^{12}\phi_{0,1}^3$ is not a holomorphic Jacobi form, we get
$$J_{6,3}(\upsilon_{\eta}^{12})=\eta^{12}\latt{\phi_{0,1}\phi_{0,2},\phi_{0,3} }_{\CC}.$$   
By comparing their first Fourier coefficients, we deduce the desired identity.
\item Proof of (\ref{c06}):
$$\Tr^{(5)}_{10,0,0,0}=\sum_{Y\in R_5}(\vartheta \lvert Y)^{10}(\tau,0)\in J_{5,0}(\upsilon_{\eta}^{6}).
$$
In view of $J_{5,0}(\upsilon_{\eta}^{6})\subset \eta^6 J_{2,0}^\w = \eta^6 M_{2}(\SL_2(\ZZ))=\{0\}$, the desired identity is proved.
\end{enumerate}
\end{proof}

Obviously, we cannot derive any holomorphic Jacobi form of weight $6$ from Theorem \ref{MTH}. In order to construct this type of Jacobi forms, we make use of the following anti-symmetric combinations of the functions in the orbit of $\vartheta$.

\begin{proposition}\label{prop02}
Let $a$, $b$, $c\in \NN$ satisfying $b\equiv 0 \m 2$ and $a+b+c\equiv 0 \m 4$. We put
$$S_{ij}=\left(\vartheta\lvert X_i\right)_{z=0}^a\left(\vartheta\lvert X_i\right)^b(\tau,z)\left(\vartheta\lvert X_i\right)^c(\tau,2z)-\left(\vartheta\lvert X_j\right)_{z=0}^a\left(\vartheta\lvert X_j\right)^b(\tau,z)\left(\vartheta\lvert X_j\right)^c(\tau,2z),$$
$$s_{ij}=\left(\vartheta\lvert X_i\right)^{a+b+c}(\tau,0)-\left(\vartheta\lvert X_j\right)^{a+b+c}(\tau,0)$$
where $X_1=\left(\frac{1}{2},0\right)$, $X_2=\left(0,\frac{1}{2}\right)$, $X_3=\left(\frac{1}{2},\frac{1}{2}\right)$.
We define
\begin{align*}
&A_{a,b,c}^{(1)}=\frac{1}{6} \left( S_{12}s_{13}s_{23}+S_{13}s_{12}s_{23}+S_{23}s_{13}s_{12}\right),\\
&A_{a,b,c}^{(2)}=\frac{1}{6} \left( S_{12}S_{13}s_{23}+S_{13}s_{12}S_{23}+S_{23}s_{13}S_{12}\right),\\
&A_{a,b,c}^{(3)}=\frac{1}{2} S_{12}S_{13}S_{23}.
\end{align*} 
Then when $a+b+c\equiv 4 \m 8$, we have
$$ A_{a,b,c}^{(i)}=1+O(q) \in J_{\frac{3(a+b+c)}{2},2ic+\frac{ib}{2}}, \quad \text{for} \; i=1,2,3.$$
When $a+b+c\equiv 0  \m 8$, we get
$$ A_{a,b,c}^{(i)} \in J_{\frac{3(a+b+c)}{2},2ic+\frac{ib}{2}}(\upsilon_{\eta}^{12}), \quad \text{for} \; i=1,2,3.$$
\end{proposition}

\begin{proof}
Denote $T=\begin{psmallmatrix}
 1 & 1 \\ 
0 & 1   
\end{psmallmatrix}$ and $S=\begin{psmallmatrix}
 0 & -1 \\ 
1 & 0   
\end{psmallmatrix}$. We find
\begin{align*}
\left(\begin{array}{c}
S_{12}\vert T \\ 
S_{13}\vert T \\ 
S_{23}\vert T
\end{array}  \right)= \left( \begin{array}{ccc}
0 & 0 & -u \\ 
0 & -u & 0 \\ 
-u & 0 & 0
\end{array}  \right) \left(\begin{array}{c}
S_{12} \\ 
S_{13} \\ 
S_{23}
\end{array}  \right),& &\left(\begin{array}{c}
S_{12}\vert S \\ 
S_{13}\vert S \\ 
S_{23}\vert S 
\end{array}  \right)= \left( \begin{array}{ccc}
-v & 0 & 0 \\ 
0 & 0 & v \\ 
0 & v & 0
\end{array}  \right) \left(\begin{array}{c}
S_{12} \\ 
S_{13} \\ 
S_{23}
\end{array}  \right),
\end{align*}
where $u=\upsilon_{\eta}(T)^{3(a+b+c)}$ and $v=\upsilon_{\eta}(S)^{3(a+b+c)}$.
It is clear that $S_{ij}$ is invariant under the action of the integral Heisenberg group. We then complete the proof by direct computation.
\end{proof}

As a corollary, we express Jacobi--Eisenstein series of weight 6 in terms of Jacobi theta functions of order 2.

\begin{corollary}
\begin{align}
&A_{2,2,0}^{(1)}:& &2E_{6,1}=(2\theta_{01}^{10}+2\theta_{01}^6\theta_{10}^4-\theta_{01}^2\theta_{10}^8)\vartheta_{01}^2+(-2\theta_{10}^{10}-2\theta_{01}^4\theta_{10}^6+\theta_{01}^8\theta_{10}^2)\vartheta_{10}^2,\\
&\label{ap07}A_{2,2,0}^{(2)}:&  &2E_{6,2}=(2\theta_{01}^{8}+\theta_{01}^4\theta_{10}^4)\vartheta_{01}^4+(2\theta_{01}^{6}\theta_{10}^2-2\theta_{01}^2\theta_{10}^6)\vartheta_{01}^2\vartheta_{10}^2-(2\theta_{10}^{8}+\theta_{01}^4\theta_{10}^4)\vartheta_{10}^4,\\
&A_{2,2,0}^{(3)}:& &2E_{6,3}=2\theta_{01}^6\vartheta_{01}^6+3\theta_{01}^4\theta_{10}^2\vartheta_{01}^4\vartheta_{10}^2-3\theta_{01}^2\theta_{10}^4\vartheta_{01}^2\vartheta_{10}^4-2\theta_{10}^6\vartheta_{10}^6-\frac{44}{61}\eta^6\vartheta^6,\\
&A_{4,0,0}^{(3)}:& &2E_6=2\theta_{01}^{12}+3\theta_{01}^{8}\theta_{10}^{4}-3\theta_{01}^{4}\theta_{10}^8-2\theta_{10}^{12}.
\end{align}
\end{corollary}

Our method can also be used to construct weak Jacobi forms of weight $0$. The following result is immediate.

\begin{proposition}\label{Prop:weak}
Let $a$, $b$, $c$ be non-negative integers satisfying $N \lvert (a+2b)$. Then the function
\begin{equation}
W_{a,b,c}^{(N)}=\sum_{Y\in R_N} \frac{(\vartheta\lvert Y)^a(\tau,z)(\vartheta\lvert Y)^b(\tau,2z)(\vartheta\lvert Y)^c(\tau,Nz)}{(\vartheta\lvert Y)^{a+b+c}(\tau,0)}
\end{equation}
is a weak Jacobi form of weight $0$ and index $\frac{a+N^2c}{2}+2b$.
\end{proposition}

When $\vartheta_{a,b}(\tau,0)\neq 0$, we write $\xi_{a,b}(\tau,z)=\vartheta_{a,b}(\tau,z)/\vartheta_{a,b}(\tau,0)$. As a direct consequence of Proposition \ref{Prop:weak}, we have the following identities.
\begin{corollary}\label{cor:weak}
\begin{align}
&W_{2,0,0}^{(2)}:& & 4(\xi_{00}^2+\xi_{01}^2+\xi_{10}^2)(\tau,z)=\phi_{0,1}(\tau,z),&\\
&W_{0,1,0}^{(2)}:& & 2(\xi_{00}+\xi_{01}+\xi_{10})(\tau,2z)=\phi_{0,2}(\tau,z),&\\
&W_{3,0,0}^{(3)}:& & \left(\xi_{\frac{1}{6},\frac{1}{6}}^3+\xi_{\frac{1}{6},\frac{1}{2}}^3+\xi_{\frac{1}{6},\frac{5}{6}}^3+\xi_{\frac{1}{2},\frac{1}{6}}^3+\xi_{\frac{1}{2},\frac{5}{6}}^3+\xi_{\frac{5}{6},\frac{1}{6}}^3+\xi_{\frac{5}{6},\frac{1}{2}}^3+\xi_{\frac{5}{6},\frac{5}{6}}^3\right)(\tau,z)=4\frac{\vartheta(\tau,2z)}{\vartheta(\tau,z)}, &\\
&W_{0,0,1}^{(3)}:& & \left(\xi_{\frac{1}{6},\frac{1}{6}}+\xi_{\frac{1}{6},\frac{1}{2}}+\xi_{\frac{1}{6},\frac{5}{6}}+\xi_{\frac{1}{2},\frac{1}{6}}+\xi_{\frac{1}{2},\frac{5}{6}}+\xi_{\frac{5}{6},\frac{1}{6}}+\xi_{\frac{5}{6},\frac{1}{2}}+\xi_{\frac{5}{6},\frac{5}{6}}\right)(\tau,3z)=\frac{\vartheta^3(\tau,2z)}{\vartheta^3(\tau,z)}, &\\
&W_{1,1,0}^{(3)}:& & \quad \xi_{\frac{1}{6},\frac{1}{6}}(\tau,z)\xi_{\frac{1}{6},\frac{1}{6}}(\tau,2z)+
\xi_{\frac{1}{6},\frac{1}{2}}(\tau,z)\xi_{\frac{1}{6},\frac{1}{2}}(\tau,2z)+\xi_{\frac{1}{6},\frac{5}{6}}(\tau,z)\xi_{\frac{1}{6},\frac{5}{6}}(\tau,2z)&\\
\nonumber& & &+\xi_{\frac{1}{2},\frac{1}{6}}(\tau,z)\xi_{\frac{1}{2},\frac{1}{6}}(\tau,2z)+\xi_{\frac{1}{2},\frac{5}{6}}(\tau,z)\xi_{\frac{1}{2},\frac{5}{6}}(\tau,2z)+\xi_{\frac{5}{6},\frac{1}{6}}(\tau,z)\xi_{\frac{5}{6},\frac{1}{6}}(\tau,2z)&\\
\nonumber& & &+\xi_{\frac{5}{6},\frac{1}{2}}(\tau,z)\xi_{\frac{5}{6},\frac{1}{2}}(\tau,2z)+\xi_{\frac{5}{6},\frac{5}{6}}(\tau,z)\xi_{\frac{5}{6},\frac{5}{6}}(\tau,2z)=\frac{1}{3}\frac{\vartheta(\tau,2z)}{\vartheta(\tau,z)}\phi_{0,1}(\tau,z). &
\end{align}
\end{corollary}

Proposition \ref{Prop:weak} may give good constructions of some Jacobi forms. For example, from the formula $2(\xi_{00}+\xi_{01}+\xi_{10})(\tau,2z)=\phi_{0,2}(\tau,z)$, it is easy to see that the function $\phi_{0,2}$ has integral Fourier coefficients because $\xi_{00}$, $\xi_{01}$ and $2\xi_{10}$ have integral Fourier coefficients. This construction has been used in \cite{GW2}.

\begin{remark}
Our approach can be generalized to the case of several variables. In other words, we can define the orbits of Jacobi forms of lattice index and use them to construct more general theta relations. For example, the well known \textit{Riemann's theta formula} (see \cite{Mu})
\begin{equation}
\prod_{i=1}^4\vartheta_{00}(\tau,z_i)+\prod_{i=1}^4\vartheta(\tau,z_i)-\prod_{i=1}^4\vartheta_{01}(\tau,z_i)-\prod_{i=1}^4\vartheta_{10}(\tau,z_i)=2\prod_{i=1}^4\vartheta(\tau,x_i),
\end{equation}
where $x_1=\frac{z_1+z_2+z_3+z_4}{2}$, $x_2=\frac{z_1+z_2-z_3-z_4}{2}$, $x_3=\frac{z_1-z_2+z_3-z_4}{2}$, $x_4=\frac{z_1-z_2-z_3+z_4}{2}$,
can be deduced by the orbit of $\prod_{i=1}^4\vartheta(\tau,z_i)$ which is a holomorphic Jacobi form of weight $2$ and index $1$ associated to the root lattice $D_4$. The interesting formula 
\begin{equation}
\prod_{i=1}^8\vartheta(\tau,z_i)+\prod_{i=1}^8\vartheta_{00}(\tau,z_i)+\prod_{i=1}^8\vartheta_{01}(\tau,z_i)+\prod_{i=1}^8\vartheta_{10}(\tau,z_i)=2\vartheta_{E_8},
\end{equation}
where 
$$
\vartheta_{E_8}(\tau,\mathfrak{z}_8)=\sum_{v\in E_8}e^{\pi i((v,v)\tau+2 (v,\mathfrak{z}_8) )},
$$
is the theta function of root lattice $E_8$, can be concluded from the orbit of $\prod_{i=1}^8\vartheta(\tau,z_i)$ which is a holomorphic Jacobi form of weight $4$ and index $1$ with respect to the root lattice $D_8$. Besides, the following formula which is a generalization of $(\ref{c03})$
\begin{equation}\label{eq:Riemann3}
\sum_{Y\in R_3} (\vartheta \lvert Y)^{3}(\tau,z_1)(\vartheta \lvert Y)^3 (\tau,z_2)=2\vartheta^3(\tau,z_1)\vartheta^3(\tau,z_2),
\end{equation}
can be derived from the orbit of $\vartheta(\tau,z_1)\vartheta(\tau,z_2)$ which is a Jacobi form of weight $1$ and index $\frac{1}{2}$ with respect to the lattice $\latt{2}\oplus \latt{2}$.
In some sense, we can view (\ref{eq:Riemann3}) as a Riemann's theta formula of order $3$. We will systematically investigate the orbits of Jacobi forms of lattice index in a separate article.
\end{remark}

\section{Applications}\label{sec:5}
In this section, we present several applications of our formulas. Firstly, as an application of (\ref{J01}), we give a simple proof of the following interesting formulae in \cite[Lemma 1]{S}.
\begin{corollary}
Assume that $N$ is odd and $v_1,v_2$ run through a set denoted by $\mathfrak{A}$ of representatives of $\frac{1}{N}(\ZZ\tau+\ZZ)/(\ZZ\tau+\ZZ)$. Then we have
\begin{align*}
\prod_{v_1 \neq v_2}[\wp(\tau,z+v_1)-\wp(\tau,z+v_2)]&=c_1 \left(\frac{\vartheta(\tau,2Nz)}{\vartheta^4(\tau,Nz)} \right)^{N^2-1}\eta^{(4N^2+3)(N^2-1)}(\tau),\\
\prod_{ \substack{v_1,v_2\neq0\\ v_1 \neq v_2}}[\wp(\tau,v_1)-\wp(\tau,v_2)]&=c_2 \Delta^{\frac{(N^2-1)(N^2-3)}{6}}(\tau),
\end{align*}
where $c_1,c_2$ are nonzero constants depending only on $N$ and $\wp$ is the Weierstrass $\wp$-function 
$$\wp(\tau,z)=\frac{1}{z^2}+\sum_{\substack{\omega\in \ZZ\tau +\ZZ \\
\omega \neq 0}}\left((z+\omega )^{-2}-\omega^{-2}  \right).$$
\end{corollary}
\begin{proof}
By the following identity
$$ \wp(\tau,z_1)-\wp(\tau,z_2)=\frac{4\pi^2 \eta^6(\tau)\vartheta(\tau,z_1+z_2)\vartheta(\tau,z_1-z_2)}{\vartheta^2(\tau,z_1)\vartheta^2(\tau,z_2)},$$
we check at once
\begin{align*}
&\prod_{v_1 \neq v_2}[\wp(\tau,z+v_1)-\wp(\tau,z+v_2)]\\
=&*\eta^{6N^2(N^2-1)}\prod_{v_1 \neq v_2}\frac{\vartheta(\tau,2z+v_1+v_2)\vartheta(\tau,v_1-v_2)}{\vartheta^2(\tau,z+v_1)\vartheta^2(\tau,z+v_2)}\\
=&*\eta^{6N^2(N^2-1)}\prod_{v_1 \neq v_2}\frac{\vartheta_{v_1+v_2}(\tau,2z)\theta_{v_1-v_2}(\tau)}{\vartheta_{v_1}^2(\tau,z)\vartheta_{v_2}^2(\tau,z)}
\end{align*}
where $*$ are constants. The notation $\vartheta_{v}$ stands for $\vartheta_{a+1/2,b+1/2}$ if $v=a\tau+b$. One can check that the set $\{v_1+v_2: v_1\neq v_2\in \mathfrak{A}  \}$ contains every element of $\mathfrak{A}$ and each element appears $N^2-1$ times. 
Similarly, we find that the set $\{v_1-v_2: v_1\neq v_2\in \mathfrak{A}  \}$ contains every element of $\mathfrak{A}$ except $0$ and each element appears $N^2$ times.
We thus prove the first formula by \eqref{J01} and \eqref{J02}. The proof of the second formula is similar. 
\end{proof}

\begin{remark}
When $N=2$, we can prove the following identities in the same way
\begin{align*}
\prod_{1\leq i<j \leq 4}[\wp(\tau,z+v_i)-\wp(\tau,z+v_j)]&=c_3\eta^{30}\left( \frac{\vartheta(\tau,4z)}{\vartheta^4(\tau,2z)}  \right)^2, \\
\prod_{1< i<j \leq 4}[\wp(\tau,v_i)-\wp(\tau,v_j)]&=c_4\eta^{12},
\end{align*}
where $c_3$, $c_4$ are constants, $v_1=0$, $v_2=1/2$, $v_3=\tau /2$ and $v_4=(\tau+1)/2$. In general, we do not have this type of formula when $N>2$ is even because the set $\{v_1+v_2: v_1\neq v_2\in \mathfrak{A}  \}$ does not satisfy the condition in the above proof.
\end{remark}

As another application, we generalize the Jacobi's derivative formula: $\vartheta^{'}(\tau,0)=\pi i \theta_{00}\theta_{01}\theta_{10}$. We first give it a simple proof.

\begin{lemma}
$$\vartheta'(\tau,0)=2\pi i \eta(\tau)^3, \; \vartheta{''}(\tau,0)=0, \; \vartheta{'''}(\tau,0)=48\pi^3 i \eta(\tau)^3G_2(\tau).$$
\end{lemma}

\begin{proof}
We consider the automorphic correction of $\frac{\vartheta(\tau,z)}{\eta(\tau)^3}$ (see \cite[Proposition 1.5]{G2}):
$$  \exp(-4\pi^2  G_2(\tau)z^2 )\frac{\vartheta(\tau,z)}{\eta(\tau)^3}= \sum_{k\geq 0}f_{k} (\tau ) z^{k}$$
where $G_2(\tau)= -\frac{1}{24} + \sum_{n \geq 1} \sigma_1 (n) q^n $ is the quasi-modular Eisenstein series of weight $2$. 
We can check that the function $f_{k}$ is a modular form of weight $k-1$ on $\SL_2(\ZZ)$. Hence $f_1(\tau)$ is a constant and $f_2=f_3=0$, which imply the lemma.
\end{proof}

It is obvious that $\vartheta_{01}'(\tau,0)=\vartheta{'}_{10}(\tau,0)=\vartheta{'}_{00}(\tau,0)=0$ because these functions are all even. Next, we calculate their second derivatives.

We take the third derivative with respect to $z$ in (\ref{ap01}):
$\vartheta(\tau,z)\vartheta_{01}\vartheta_{10}\vartheta_{00}=\eta^3 \vartheta(\tau,2z)$, and set $z=0$ in the obtained identity. Then we get
\begin{equation}\label{ap02}
-\pi^2E_2(\tau)=\frac{\vartheta{'''}(\tau,0)}{\vartheta{'}(\tau,0)}=\frac{\vartheta{''}_{01}(\tau,0)}{\theta_{01}(\tau)}+\frac{\vartheta{''}_{10}(\tau,0)}{\theta_{10}(\tau)}+\frac{\vartheta{''}_{00}(\tau,0)}{\theta_{00}(\tau)}.
\end{equation}
By taking the second derivative with respect to $z$ and setting $z=0$ in (\ref{c01}), we obtain
\begin{equation}\label{ap03}
\vartheta{''}_{01}(\tau,0)\theta_{01}(\tau)^3+\vartheta{''}_{10}(\tau,0)\theta_{10}(\tau)^3-\vartheta{''}_{00}(\tau,0)\theta_{00}(\tau)^3=0.
\end{equation}
By using (\ref{c00}) and (\ref{c01}), we can show
\begin{equation}
\vartheta^2=\frac{\theta_{01}^2\vartheta_{10}^2-\theta_{10}^2\vartheta_{01}^2}{\theta_{00}^2},
\end{equation}
because they all belong to $J_{1,1}(\upsilon_{\eta}^{6})$. By taking the second derivative with respect to $z$ and setting $z=0$ in the above formula, we get
\begin{equation}\label{ap04}
\frac{\vartheta{''}_{01}(\tau,0)}{\theta_{01}(\tau)}-\frac{\vartheta{''}_{10}(\tau,0)}{\theta_{10}(\tau)}=\pi^2\theta_{00}(\tau)^4.
\end{equation}
By (\ref{ap02}), (\ref{ap03}) and (\ref{ap04}), we get
\begin{align}
& \frac{\vartheta{''}_{01}(\tau,0)}{\theta_{01}(\tau)}=-\frac{\pi^2}{3} E_2+\frac{\pi^2}{3}(\theta_{00}^4+\theta_{10}^4),\\
& \frac{\vartheta{''}_{10}(\tau,0)}{\theta_{10}(\tau)}=-\frac{\pi^2}{3} E_2-\frac{\pi^2}{3}(\theta_{00}^4+\theta_{01}^4),\\
& \frac{\vartheta{''}_{00}(\tau,0)}{\theta_{00}(\tau)}=-\frac{\pi^2}{3} E_2+\frac{\pi^2}{3}(\theta_{01}^4-\theta_{10}^4).
\end{align}

Besides, if we take a special value on $z$ in a given identity, we will get an identity that relates special values of Jacobi--Eisenstein series to theta constants. For example, by taking $z=\frac{1}{2}$ in (\ref{ap05}) and (\ref{ap06}), we get
\begin{align*}
&\theta_{00}^4\theta_{01}^4=E_{4,2}\Big(\tau,\frac{1}{2}\Big)=E_{4,4}\Big(\tau,\frac{1}{2}\Big),& &\theta_{10}^8=E_4-E_{4,2}\Big(\tau,\frac{1}{2}\Big).&
\end{align*}
Taking $z=\frac{1}{2}$ in (\ref{ap07}), we have
$$ (\theta_{00}^4+\theta_{01}^4)\theta_{10}^8=E_{6,2}\Big(\tau,\frac{1}{2}\Big)-E_6.$$ 
Note that more this type of identities can be found in \cite{GW1}.

Finally, we remark that our approach can also be used to construct Jacobi forms on congruence subgroups. For instance, if we fix $X=(\frac{1}{2},\frac{1}{2})$ and consider the orbit of $\vartheta$ with respect to the set $X\cdot \Gamma_0(2)/\ZZ^2=\{(\frac{1}{2},0),(\frac{1}{2},\frac{1}{2})\}$, then we can construct 
$$\frac{1}{2}\left[\theta_{00}^2(\tau)\vartheta_{00}^2(\tau,z)+\theta_{01}^2(\tau)\vartheta_{01}^2(\tau,z) \right]=1+q(\zeta^{\pm 2}+8\zeta^{\pm 1}+6)+O(q^2)\in J_{2,1}(\Gamma_0(2)),$$
which is a holomorphic Jacobi form of weight 2 and index 1 on $\Gamma_0(2)$. In \cite{BSZ}, the authors gave explicit formulas for Jacobi--Eisenstein series of weight 2 and index 1 on $\Gamma_0(2)$. For any prime number $p$, they proved that the function 
\begin{equation}
E_{2,1,p}(\tau,z)=\sum_{\substack{ n\in\NN, r\in \ZZ\\ 4n\geq r^2}} H^{(p)}(4n-r^2) q^n\zeta^r
\end{equation}
is a holomorphic Jacobi form of weight 2 and index 1 on $\Gamma_0(p)$. Here, for $D\geq 0$, $H(D)$ is the Hurwitz class number and
\begin{equation}
H^{(p)}(D)=H(p^2D)-pH(D).
\end{equation}
By $\dim J_{2,1}(\Gamma_0(2))=1$, we obtain 
\begin{equation}
\theta_{00}^2(\tau)\vartheta_{00}^2(\tau,z)+\theta_{01}^2(\tau)\vartheta_{01}^2(\tau,z)=24\sum_{\substack{ n\in\NN, r\in \ZZ\\ 4n\geq r^2}} H^{(2)}(4n-r^2) q^n\zeta^r.
\end{equation}
Let $z=0$, we have
\begin{equation}
\theta_{00}^4(\tau)+\theta_{01}^4(\tau)=24\sum_{n\in \NN} \sum_{\substack{ r\in \ZZ \\ \abs{r}\leq 2\sqrt{n}}} H^{(2)}(4n-r^2) q^n.
\end{equation}
Let $z=\frac{1}{2}$, we obtain 
\begin{equation}
\theta_{00}^2(\tau)\theta_{01}^2(\tau)=12\sum_{n\in \NN} \sum_{\substack{ r\in \ZZ \\ \abs{r}\leq 2\sqrt{n}}} (-1)^r H^{(2)}(4n-r^2) q^n.
\end{equation}
Similarly, by taking $z=\frac{\tau}{2}$ and $z=\frac{\tau+1}{2}$, we get formulas expressing $\theta_{00}^2\theta_{10}^2$ and $\theta_{01}^2\theta_{10}^2$, respectively. 

Moreover, it is obvious that $E_{2,1,p}(\tau,0)\in M_2(\Gamma_0(p)) $ and
\begin{equation}
pE_2(p\tau)-E_2(\tau)=p-1+24\sum_{n\geq 1}\sum_{\substack{ d\vert n, d>0\\ (d,p)=1}} d q^n \in M_2(\Gamma_0(p)) .
\end{equation}
Hence, when $\dim M_2(\Gamma_0(p))=1$ (such as $p=2,3,5,7,13$), we have 
\begin{equation}
\sum_{\substack{ r\in \ZZ \\ \abs{r}\leq 2\sqrt{n}}} H^{(p)}(4n-r^2)=2\sum_{\substack{ d\vert n, d>0\\ (d,p)=1}} d.
\end{equation}

\bigskip

\noindent
\textbf{Acknowledgements.} The first author was supported by the HSE University Basic Research Program and by the PRCI SMAGP 
(ANR-20-CE40-0026-01). 

\bibliographystyle{amsplain}

\end{document}